\documentclass[letterpaper,10pt,conference]{ieeeconf}  
\IEEEoverridecommandlockouts 
\usepackage{amsmath,mathptmx}
\usepackage[pdftex]{graphicx}
\usepackage[caption=false,font=footnotesize]{subfig}
  \graphicspath{{./pics/}}
  \DeclareGraphicsExtensions{.pdf}
\newcommand{\eps}{\varepsilon}
\newcommand{\supC}{\mbox{$\sup {\rm C}$}}

\newcommand{\suptwoCC}{\mbox{$\sup {\rm 2cC}$}}

\newtheorem{theorem}{Theorem}
\newtheorem{problem}[theorem]{Problem}
\newtheorem{lem}[theorem]{Lemma}
\newtheorem{corollary}[theorem]{Corollary}
\newtheorem{definition}[theorem]{Definition}

\newtheorem{rem}[theorem]{Remark}

\title{\LARGE \bf Decentralized Supervisory Control with Communicating Supervisors Based on Top-Down Coordination Control}
\author{Jan Komenda and Tom{\' a}{\v s} Masopust
  \thanks{J. Komenda and T. Masopust are with the Institute of Mathematics,
          Academy of Sci. of the Czech Republic,
          {\v Z}i{\v z}kova 22, 616 62 Brno, Czech Rep.
          {\tt\small komenda@math.cas.cz, masopust@math.cas.cz}
  }%
}
\begin{document}

\maketitle
\thispagestyle{empty}
\pagestyle{empty}

\begin{abstract}
  In this paper we present a new approach to decentralized supervisory control of large automata with communicating supervisors. We first generalize the recently developed top-down architecture of multilevel coordination control with a hierarchical structure of groups of subsystems, their respective coordinators and supervisors. Namely, in the case where the equivalent conditions for achieving a specification language fail to be satisfied, we propose sufficient conditions for a distributed computation of the supremal achievable sublanguage. We then apply the obtained constructive results of multilevel coordination control to decentralized supervisory control with communication, where local supervisors of subsystems within a group communicate with each other via the coordinator of the group. Our approach is illustrated by an example.
\end{abstract}

\section{Introduction}
  Decentralized supervisory control has been introduced in~\cite{RW92} aiming at decreasing the computational complexity of supervisory control of large automata without a product structure based on a single (centralized) supervisor. As it turned out that {\em coobservability\/} of a specification language, the necessary condition to achieve the specification in decentralized supervisory control, is too strong, two different approaches have been proposed. The first one is based on communication between supervisors to achieve coobservability with respect to observable alphabets enriched by communicated event occurrences, cf. \cite{BL00} and \cite{RR00}. The second one consists in proposing new and more general decentralized supervisory control architectures that lead to weaker notions of coobservability. The original notion of coobservability~\cite{RW92} has been called conjunctive and permissive (C\,\&\,P), while an alternative architecture called disjunctive and antipermissive (D\,\&\,A) has been proposed in~\cite{YL03} together with their combination. Among weaker concepts of coobservability, one should cite a general architecture combining both architectures~\cite{YL03} or even more general architectures with possible several levels of inferencing leading to conditional versions of coobservability~\cite{YL04}.

  Coordination control \cite{KMvS14} can be seen as a trade-off between a purely local (modular) control synthesis that often fails in achieving sufficiently permissive supervisors and leads to blocking, and a global control synthesis that is too expensive from the computational complexity viewpoint. We have extended coordination control to the multilevel setting with a hierarchical structure of groups of subsystems, their respective coordinators, and supervisors in~\cite{KMvS13}. The top-down approach proposed therein along with the corresponding notions of conditional decomposability and conditional controllability enables to compute the supervisors only at the lowest level. Unlike centralized coordination, supervisors of the subsystems only communicate within groups on the same level of hierarchy via group coordinators, which are located on the next upper level of the hierarchy.

  In~\cite{KM13}, we have proposed a generic approach that consists in applying the results of coordination control of automata with a synchronous-product structure to decentralized control of automata without an explicit product structure. It is based on the over-approximation of the automaton without a product structure by a product of its natural projections to the alphabet of local supervisors. However, a communication structure is not interesting, because all local supervisors communicate with each other via a coordinator.

  In this paper we will benefit from recent results of multilevel coordination control. First, earlier existential results from~\cite{KMvS13} are extended to a construction procedure that computes the least restrictive solution for the top-down architecture of multilevel coordination control. These results are then applied to the original decentralized control problem, which is solved using an underlying communication scheme, where local supervisors of subsystems within a group communicate with each other via a coordinator of the group. The optimal solution obtained by coordination control leads to a (possibly non-optimal) solution of the the original decentralized control problem.

\section{Existential results of multilevel coordination control}\label{existential}
  Basic notions and notational conventions are first recalled. The free monoid of words (strings) over an alphabet (event set) $A$ is denoted by $A^*$. A language is a subset of $A^*$. The prefix closure of a language $L\subseteq A^*$ is the set of all prefixes of all its words and is denoted by $\overline{L}=\{w\in A^* \mid \text{there exists } v\in A^* \text{ such that } wv\in L\}$; $L$ is prefix-closed if $L=\overline{L}$. 
  
  In this paper, any languages under consideration are assumed to be prefix-closed.

  A {\em generator\/} is a quadruple $G=(Q,A,f,q_0)$ consisting of a finite set of {\em states} $Q$, a finite {\em event set }$A$, a {\em partial transition function} $f: Q \times A \to Q$, and an {\em initial state} $q_0 \in Q$. The function $f$ can be extended in the standard way to strings, i.e. $f:Q \times A^*\to Q$. We recall that $L(G) = \{s\in A^* \mid f(q_0,s)\in Q\}$ is called the {\em generated language\/} of $G$. 
  
  A {\em controlled generator\/} over an alphabet $A$ is a triple $(G,A_c,\Gamma)$, where $G$ is a generator over $A$, $A_c\subseteq A$ is a set of {\em controllable events\/}, $A_{u} = A \setminus A_c$ is the set of {\em uncontrollable events\/}, and $\Gamma = \{\gamma \subseteq A \mid A_{u} \subseteq \gamma\}$ is the {\em set of control patterns}. A {\em supervisor\/} for a controlled generator $(G,A_c,\Gamma)$ is a map $S:L(G) \to \Gamma$. The {\em closed-loop system\/} associated with the controlled generator $(G,A_c,\Gamma)$ and the supervisor $S$ is defined as the minimal language $L(S/G)$ such that $\eps\in L(S/G)$ and, for any $s\in L(S/G)$ with $sa\in L(G)$ and $a\in S(s)$, $sa$ belongs to $L(S/G)$.
  
  A language $K\subseteq A^*$ is {\em controllable} with respect to $L$ and $A_u$ if $KA_u\cap L\subseteq K$.

  A {\em projection} $P: A^* \to B^*$, for $B\subseteq A$, is a homomorphism defined as $P(a)=\eps$, for $a\in A\setminus B$, and $P(a)=a$, for $a\in B$. The {\em inverse image} of $P$, denoted by $P^{-1} : B^* \to 2^{A^*}$, is defined as $P^{-1}(w)=\{s\in A^* \mid P(s) = w\}$. These definitions can be extended to languages. For alphabets $A_i$, $A_j$, $A_\ell \subseteq A$, we use $P^{i+j}_{\ell}$ to denote the projection from $(A_i\cup A_j)^*$ to $A_\ell^*$. If $A_i\cup A_j=A$, we simply write $P_\ell$. Moreover, $A_{i,u}=A_i\cap A_u$ denotes the set of locally uncontrollable events. For a generator $G$ and a projection $P$, $P(G)$ denotes the minimal generator such that $L(P(G))=P(L(G))$. The reader is referred to~\cite{CL08} for a construction.

  Let $G$ be a generator over an alphabet $A$. 
  Given a specification $K\subseteq L(G)$, the aim of supervisory control is to find a supervisor $S$ such that $L(S/G)=K$. Such a supervisor exists if and only if $K$ is controllable with respect to $L(G)$ and $A_u$, 
  see~\cite{CL08}.
  
  The synchronous product of languages $L_i\subseteq A_i^*$, for $i=1,\dots ,n$, is defined as $\|_{i=1}^n L_i= \cap_{i=1}^n P_i^{-1}(L_i) \subseteq A^*$, where $A = \cup_{i=1}^n A_i$, and $P_i: A^*\to A_i^*$ are projections to local alphabets. The corresponding synchronous product of generators $G_i$ (see~\cite{CL08} for definition) satisfies $L(\|_{i=1}^n G_i) = \|_{i=1}^n L(G_i)$. 

  A projection $Q:A^* \to B^*$ is an {\em $L$-observer} for a language $L\subseteq A^*$ if, for every $t\in Q(L)$ and $s\in L$, $Q(s)$ is a prefix of $t$ implies that there exists $u\in A^*$ such that $su\in L$ and $Q(su)=t$, cf.~\cite{WW96}. 
  
  We will need the following obvious lemma and results below.
  
  \begin{lem} 
  \label{obvious}
    For any language $L\subseteq A^*$ and projections $P_1: A^* \to B_1^*$ and $P_2: A^* \to B_2^*$ with $B_2\subseteq B_1\subseteq A$, it holds that $P_1(L) \| P_2(L)= P_1(L)$. \hfill\QED
  \end{lem} 

  
  \begin{lem}[Lemma~4.3 in~\cite{FLT}]\label{LeiFeng}
    Let $K\subseteq L$ be controllable with respect to $L$ and $A_u$. 
    If the natural projection $P:A^*\to A_o^*$ is an $L$-observer and OCC for $L$,
    then $P(K)$ is controllable with respect to $P(L)$ and $A_u\cap A_o$. 
    \hfill\QED
  \end{lem}

  \begin{lem}[Proposition~4.6 in \cite{FLT}]\label{feng}
    For $i=1,\dots ,n$, let $K_i\subseteq L_i$ be controllable with respect to $L_i\subseteq A_i^*$ and $A_{i,u}$, then $\|_{i=1}^n K_i$ is controllable with respect to $\|_{i=1}^n L_i$ and $\cup_{i=1}^n A_{i,u}$. \hfill\QED
  \end{lem}

  Now we recall the basic notions of coordination control~\cite{KMvS14}. 
  A language $K$ over $\cup_{i=1}^{n}A_i$ is {\em conditionally decomposable with respect to alphabets $(A_i)_{i=1}^{n}$ and $A_k$}, where $\cup_{1\le i,j\le n}^{i\neq j} (A_i\cap A_j) \subseteq A_k\subseteq \cup_{i=1}^{n} A_j$, if 
  \[
    K =\ \parallel_{i=1}^{n} P_{i+k} (K)\,,
  \]
  for projections $P_{i+k}$ from $\cup_{j=1}^{n} A_j$ to $A_i\cup A_k$, for $i=1,\ldots,n$.
  Alphabet $A_k$ is referred to as a coordinator alphabet and satisfies the {\em conditional independence} property, that is, $A_k$ includes all shared events: $\cup_{1\le i,j\le n}^{i\neq j} (A_i\cap A_j) \subseteq A_k$. 
  This has the following well known impact.
  \begin{lem}[\cite{FLT}]\label{lemma:Wonham}
    Let $P_k : A^*\to A_k^*$ be a projection, and let $L_i$ be a language over $A_i$, for $i=1,\dots ,n$, and $\cup_{1\le i,j\le n}^{i\neq j} (A_i\cap A_j) \subseteq A_k$. Then $P_k(\|_{i=1}^n L_i)=\|_{i=1}^n P_k(L_i)$.
    \hfill\QED
  \end{lem}


  The idea of coordination control is to first construct a supervisor $S_k$ such that the closed-loop system $L(S_k/G_k)$ satisfies the "coordinator part" of the specification given by $P_k(K)$ and then local supervisors $S_i$ for plants $G_i \| (S_k/G_k)$, for $i=1,\dots ,n$, such that the closed-loop system $L(S_i/ [G_i \| (S_k/G_k) ])$ satisfies the corresponding part of the specification given by $P_{i+k}(K)$. Conditional controllability along with conditional decomposability form an equivalent condition for a language to be achieved by the closed-loop system within our coordination control architecture, see below.
  
  A language $K\subseteq L(G_1\| \dots \|G_n \| G_k)$ is {\em conditionally controllable\/} for generators $G_1,\dots,G_n$ and a coordinator $G_k$ and uncontrollable alphabets $A_{1,u},\dots ,A_{2,u}$, and $A_{k,u}$ if 
  (1) $P_k(K)$ is controllable with respect to $L(G_k)$ and $A_{k,u}$ and
  (2) $P_{i+k}(K)$ is controllable with respect to $L(G_i) \parallel P_k(K)$ and $A_{i+k,u}=A_{i+k,u}=(A_i\cup A_k)\cap A_u$, for $i=1,\dots,n$.

  Recall that every conditionally controllable and conditionally decomposable language is controllable~\cite{KMvS14}. The main existential result of~\cite{KMvS14} states that for a specification $K\subseteq A^*$ that is conditionally decomposable, there exist supervisors $S_1, \dots ,S_n$, and $S_k$ such that 
  $\parallel_{i=1}^n L(S_i/[G_i \| (S_k/G_k)]) =  K$ if and only if $K$ is conditionally controllable.

\section{Constructive results of two-level coordination control}\label{sec:sup2cc}
  In this section we assume that $G=G_1 \| G_2 \|\dots \|G_n$ and that the subsystems are organized into $m$ groups $I_j$, for $j=1,\dots ,m$. 
  The notation \[A_{I_r}=\bigcup\nolimits_{i\in I_r} A_i\] is used in the paper. Here $P_{I_r}$ denotes the projection $P_{I_r}:A^*\to A_{I_r}^*$. The notation for projection to extended group events $P_{I_r+k}: A^*\to (A_k\cup A_{I_r})^*$ should be self-explanatory. We have introduced the corresponding notion of conditional decomposability in~\cite{KMvS13}.

  \begin{definition}[Two-level conditional decomposability]\label{mcd}$ $\\
    A language $K\subseteq A^*$ is called {\em two-level conditionally decomposable\/} with respect to alphabets $A_1,\ldots,A_n$, high-level coordinator alphabet $A_{k}$, and low-level coordinator alphabets $A_{k_1},\dots A_{k_m}$ if
    \begin{align*} 
      K = \parallel_{r=1}^m P_{I_r+k} (K) && \text{ and } && P_{I_r+k} (K) &= \parallel_{j\in I_r} P_{j+k_r+k} (K)\,
    \end{align*}
    for $r=1,\dots,m$.
    \hfill\QEDopen
  \end{definition} 
  
  \begin{rem} \label{simple} 
    Unlike the original approach in~\cite{KMvS13} we propose the following simplification. Instead of using both low-level coordinator alphabets $A_{k_j}$, $j=1, \dots, m$, and high-level coordinator alphabet $A_{k}$, we will only use low-level coordinator alphabets $A_{k_j}$, $j=1, \dots, m$. We recall that $A_{k}$ contains only events shared between different groups of subsystems, i.e. $A_k\supseteq \bigcup\nolimits_{k,\ell\in\{1, \dots ,m\}}^{k\not =l} (A_{I_k}\cap A_{I_\ell})$, which is typically a much smaller set than the set of shared events (between two or more subsystems). Thus, for $j=1, \dots, m$, we define new low-level coordinator alphabets $A_{k_j}:=A_k\cup A_{k_j}$ by putting into the alphabets of group coordinators $A_{k_j}$ also events from the global coordinator set (if this is nonempty). This is possible, because only prefix-closed languages are used, which means that no high-level coordinators for nonblocking are needed. We recall that in the  prefix-closed case the coordinators are actually determined by the corresponding alphabets from Definition~\ref{mcd} as a projection of the plant to these alphabets. This simplification is used because this paper is technically involved.
 \end{rem} 
 
 \begin{problem}[Two-level coordination control problem]
  \label{problem:m-controlsynthesis} 
    Let generators $G_1,\dots,G_n$ be over alphabets $A_1,\dots,A_n$, respectively, and consider the modular system as their synchronous product $G=G_1\|\dots \| G_n$ along with the two-level hierarchical structure of subsystems organized into groups $I_j$, $j=1,\dots ,m\leq n$, on the low level. The synchronous products $\parallel_{i\in I_j} G_i$, $j=1,\dots ,m$, then represent the $m$ high-level systems. 
    Coordinators $G_{k_j}$ are associated to groups of subsystems $\{G_i \mid i\in I_j\}$, $j=1,\dots, m$.  The two-level coordination control problem consists in synthesizing supervisors $S_i$ for any low-level systems $G_i$, $i=1,\ldots,n$, and higher-level supervisors $S_{k_j}$ supervising the group coordinator $G_{k_j}$, $j=1,\dots, m$, such that the specification is met by the closed-loop system. 
    Then the two-level coordinated and supervised closed-loop system is 
    \begin{flalign*}
      && \parallel_{j=1}^m  \parallel_{i\in I_j} L(S_i/  [G_i\parallel (S_{k_j}/G_{k_j})])\,. && \triangleleft
    \end{flalign*}
  \end{problem}
  \medskip

  For a specification $K$, the coordinator $G_{k_j}$ of the $j$-th group of subsystems $\{G_i \mid i\in I_j\}$ is computed as follows.
  \begin{enumerate}
    \item Set $A_{k_j} = \bigcup_{k,\ell\in I_j}^{k\neq \ell} (A_k\cap A_\ell)$ to be the set of all shared events of systems from the group $I_j$.
    \item Extend $A_{k_j}$ so that $P_{I_r+k}(K)$ is conditional decomposable with respect to $(A_i)_{i\in I_j}$ and $A_{k_j}$, for instance using a method described in~\cite{SCL12}.
    \item Let coordinator $G_{k_j}= \|_{i= 1}^{n} P_{k_j}(G_i)$.
  \end{enumerate}
  
  With the definition that $A_k\subseteq A_{k_j}$ described in Remark~\ref{simple}, we can simplify $L(G_k) \| L(G_{k_j})$ of~\cite{KMvS13} to $L(G_{k_j})$. Indeed, by our choice of coordinators $L(G_k) \| L(G_{k_j}) = P_k(L) \| P_{k_j}(L) = P_{k_j}(L) = L(G_{k_j})$. Therefore, instead of both low-level coordinators $G_{k_j}$, $j=1,\dots ,m$, for subsystems belonging to individual groups $\{G_i \mid i\in I_j\}$ and high-level coordinators $G_k$ that coordinate the different groups, we are using only low-level (group) coordinators $G_{k_j}$, but over larger alphabets compared to~\cite{KMvS13}. 

  
  Since the only known condition ensuring that the projected generator is smaller than the original one is the observer property~\cite{WW96} we might need to further extend alphabets $A_{k_j}$ so that projection $P_{k_j}$ is an $L(G_i)$-observer, for all $i\in I_j$. 
  
  We assume that the specification is prefix-closed, hence the blocking issue is not considered in this paper. Blocking can be handled using coordinators for nonblockingness studied in~\cite{KMvS14}.

  The key concept is the following.
  \begin{definition}[\cite{KMvS13}]\label{def:2-conditionalcontrollability}
    Consider the setting and notations of Problem~\ref{problem:m-controlsynthesis}, and let $G_k$ be a coordinator. A language $K\subseteq L(\|_{i=1}^{n} G_i)$ is {\em two-level conditionally controllable\/} with respect to generators $G_1,\dots,G_n$, local alphabets $A_1,\ldots,A_n$, low-level coordinator alphabets $A_{k_1},\dots A_{k_m}$, and uncontrollable alphabet $A_{u}$ if
    \begin{enumerate}
      \item\label{cc1} $P_{k_j}(K)$ is controllable with respect to  $L(G_{k_j})$ and $A_{k_j,u}$,
      \item\label{cc2} for $j=1,\dots,m$ and $i\in I_j$, $P_{i+k_j}(K)$ is controllable with respect to  $L(G_i) \parallel P_{k_j}(K)$ and $A_{i+k_j,u}$.
      \hfill\QEDopen
    \end{enumerate}
  \end{definition}
  
  Note that we have simplified the original version of two-level conditional controllability from \cite{KMvS13} by replacing the composition $L(G_k) \| L(G_{k_j})$ by $L(G_{k_j})$ as discussed in Remark \ref{simple}. For a future reference we will say that $K$ is two-level conditionally controllable with respect to $G_i$, $i\in I_l$, and $G_{k_l}$, i.e. several items are omitted (but should be clear from the problem formulation).
 
  The following lemma shows how to construct a two-level conditional controllable language as the synchronous composition of conditionally controllable languages for groups. 
  \begin{lem}\label{cc2cc}
    For all $l=1,\dots ,m$, let the languages $M_l\subseteq A_{I_l}^*$ be conditionally controllable with respect to $G_i$, $i\in I_l$, and $G_{k_l}$, and conditionally decomposable with respect to alphabets $A_i$, $i\in I_l$, and $A_{k_l}$, and $A_{k_l} \supseteq A_k \supseteq \bigcup\nolimits_{k\neq \ell} (A_{I_k}\cap A_{I_\ell})$. If for all $l=1,\dots ,m$, $P_k^{k_l}$ is a $L_{k_l}$-observer and OCC for $P_{k_l}(M_l)$, then $\parallel_{l=1}^m M_l$ is two-level conditionally controllable with respect $G_i$, $i\in I_l$, and $G_{k_l}$, $l=1,\dots ,m$.
  \end{lem}
  \begin{proof}
    First of all, note that $M = \|_{l=1}^m M_l$ is two-level conditionally decomposable. Indeed, $M=\|_{j=1}^m P_{I_j+k}(M)$, for $A_k=\emptyset$, (and hence for any $A_k\supseteq \emptyset$ as well). 
    
    For item~\ref{cc1} of the definition, we have, by Lemma \ref{lemma:Wonham}, $P_{k_j}(\|_{l=1}^m M_l)=\|_{l=1}^m P_{k_j}(M_l)$, because $A_{k_j}\supseteq A_k \supseteq \cup_{k\neq \ell} (A_{I_k}\cap A_{I_\ell})$.
    We obtain
    $\|_{l=1}^m P_{k_j}(M_l)=P_{k_j}(M_j)\parallel \|_{l=1,\dots ,m}^{l\neq j} P_{k_j}(M_l)$.
    Now, $P_{k_j}(M_j)$ is controllable with respect to $L(G_{k_j})$ and $A_{k_j,u}$
    (by the assumption that $M_j$ is conditionally controllable with respect to 
    $G_i$, $i\in I_j$, and $G_{k_j}$). Furthermore,
    for any $l=1, \dots ,m$, $l\neq j$, $P_{k_j}(M_l)=P_k(M_l)$, because 
    $M_l \subseteq A_{I_l}^*$, 
    $A_{k_j} \subseteq A_{I_j}\cup A_k$, 
    $A_{I_j} \cap A_{I_l}\subseteq A_k$, whence
    $A_{k_j}\cap A_{I_l}\subseteq A_k$. 
    Note that 
    $P_k(M_l)=P_kP_{k_l}(M_l)$.
    Since $M_l$, $l\neq j$, are conditionally controllable with respect to 
    $G_i$, $i\in I_l$, and $G_{k_l}$, we have that $P_{k_l}(M_l)$ is controllable with respect to $L(G_{k_l})$ and $A_{k_l,u}$. By Lemma~\ref{LeiFeng} and observer and OCC
assumptions, it follows that
    $P_kP_{k_l}(M_l)=P_k(M_l)=P_{k_j}(M_l)$ is controllable with respect to $P_kL(G_{k_l})$ and $A_{k,u}=A_k\cap A_{k_l,u}$. 
    It follows by Lemma~\ref{feng} that
    $\|_{l=1}^m P_{k_j}(M_l)=P_{k_j}(M_j)\parallel \|_{l=1,\dots ,m}^{l\neq j} P_{k_j}(M_l)$
    is controllable with respect to $L(G_{k_j})\parallel \|_{l=1,\dots ,m}^{l\neq j} P_kL(G_{k_l})$ and $A_{k_j,u}$.
    We recall at this point that
    $L(G_{k_l})= \|_{i=1}^{n} P_{k_l}(L(G_i))$, hence
    $P_kL(G_{k_l}) = \|_{i=1}^{n} P_{k}(L(G_i))$.
    It follows from $A_k\subseteq A_{k_j}$ and Lemma~\ref{obvious} that for any $l=1,\dots ,m$, $l\neq j$,
    we have $L(G_{k_j}) \parallel P_kL(G_{k_l}) = \|_{i=1}^{n} P_{k_j}(L(G_i)) \parallel \|_{i=1}^{n} P_{k}P_{k_j}(L(G_i)) =  \|_{i=1}^{n} P_{k_j}(L(G_i)) = L(G_{k_j})$.
    Therefore, $P_{k_j}(\|_{j=1}^m M_j)$ is controllable with respect to $L(G_{k_j})$ and $A_{k_j,u}$, which was to be shown for item~\ref{cc1} of the two-level conditional controllability.

    For item~\ref{cc2} we have to show that for all $j=1, \dots, m$
    and for all $i\in I_j$,
    $P_{i+k_j}(\|_{l=1}^m M_l)$ is controllable with respect to 
    $L(G_i)\|P_{k_j}(\|_{l=1}^m M_l)$ and $A_{i+k_j,u}$.
    First of all,
    $P_{i+k_j}(\|_{l=1}^m M_l)=\|_{l=1}^m P_{i+k_j}(M_l)$ by
    the same reason as for item~\ref{cc1}: we project to the event set
    $A_{i+k_j}$ that is even larger than $A_{k_j}$.
    Now,
    $\|_{l=1}^m P_{i+k_j}(M_l)=P_{i+k_j}(M_j)\parallel 
    \|_{l=1,\dots ,m}^{l\neq j}P_{i+k_j}(M_l)$.
    Since $M_j$ is conditionally controllable (for group $I_j$),
    by item 2), $P_{i+k_j}(M_j)$ is controllable with respect to  $L(G_i) \| P_{k_j}(M_j)$
    and $A_{i+k_j,u}$.
    Obviously, $\|_{l=1,\dots ,m}^{l\neq j}P_{i+k_j}(M_l)$ is
    controllable with respect to itself and corresponding uncontrollable event set.
    Altogether we obtain from Lemma~\ref{feng} that
    $\|_{l=1}^m P_{i+k_j}(M_l) = P_{i+k_j}(M_j) \parallel \|_{l=1,\dots ,m}^{l\neq j} P_{i+k_j}(M_l)$ is
    controllable with respect to $L(G_i) \| P_{k_j}(M_j) \parallel \|_{l=1,\dots ,m}^{l\neq j} P_{i+k_j}(M_l) 
    = L(G_i) \parallel P_{k_j} (\|_{l=1}^m M_l)$, which completes the proof.
  \end{proof}

  The main existential result of multilevel coordination control is now recalled from \cite{KMvS13}. 
  
  \begin{theorem}\label{th:2-controlsynthesissafety}
    Consider the setting of Problem~\ref{problem:m-controlsynthesis}  (in particular $K$ is two-level conditionally decomposable  with respect to local alphabets $A_1,\ldots,A_n$, high-level coordinator alphabet $A_{k}$, and low-level coordinator alphabets $A_{k_1},\dots A_{k_m}$). There exist supervisors for low-level systems $S_i$, $i\in I_j$, within any group of low-level systems $\{G_i \mid i\in I_j\}$, $j=1,\dots, m$, and supervisors $S_{k_j}$, $j=1,\dots, m$, for low-level coordinators such that 
    \begin{align}\label{eq:controlsynthesissafety}
      \parallel_{j=1}^m  \parallel_{i\in I_j} L(S_i/  G_i\parallel (S_{k_j}/ G_{k_j}))=  K
    \end{align}
    if and only if $K$ is two-level conditionally controllable as defined in Definition~\ref{def:2-conditionalcontrollability}.
    \hfill\QED
  \end{theorem}

  In the last section we have recalled two-level coordination control framework with the main existential result. A natural question is what to do if the specification fails to satisfy the necessary and sufficient conditions for being achievable. We recall from~\cite{KMvS14} that in the case specification $K$ fails to be conditionally controllable, the supremal conditionally controllable sublanguage always exists and can be computed in a distributive way. First, we show that two-level conditional controllability is closed under language unions as well.

  \begin{theorem}\label{existence}
    Two-level conditional controllability is closed under language unions, that is, supremal two-level conditional controllable languages always exist.
  \end{theorem}
  \begin{proof}
    We have to show that an arbitrary union of two-level conditionally-controllable sublanguages of a language is two-level conditionally controllable. Let $I$ be an index set, and let $K_i$, $i\in I$, be two-level conditionally controllable sublanguages of a language $K\subseteq A^*$. We will show that $\cup_{i\in I} K_i$ is also two-level conditionally controllable.

    According to~\ref{cc1} of Definition~\ref{def:2-conditionalcontrollability}, we have to show that $P_{k_j}(\cup_{i\in I} K_i)$ is controllable with respect to $L(G_{k_j})$ and $A_{k_j,u}$. This is straightforward. Indeed, for all $j=1,\dots,m$, $P_{k_j}(\cup_{i\in I} K_i)A_{k_j,u} \cap L(G_{k_j}) = \cup_{i\in I} (P_{k_j}(K_i)A_{k_j,u} \cap L(G_{k_j}))\subseteq \cup_{i\in I} P_{k_j}(K_i)$, because all $K_i$, $i\in I_j$, are two-level conditionally controllable and hence satisfy item~\ref{cc1}. Concerning item \ref{cc2}, let us show that for $j=1,\dots,m$ and $i\in I_j$, $P_{i+k_j}(\cup_{i\in I} K_i)$ is controllable with respect to $L(G_i) \parallel P_{k_j}(\cup_{i\in I} K_i)$ and $A_{i+k_j,u}$. Note that $P_{i+k_j}(\cup_{i\in I} K_i) A_{i+k_j,u}  \cap L(G_i) \parallel P_{k_j} (\cup_{j\in I} K_j) = \cup_{i\in I} \cup_{j\in I} [P_{i+k_j}(K_i)A_{i+k_j,u}  \cap L(G_i) \parallel P_{k_j}( K_j)]$. Let us take $s\in P_{i+k_j}(K_i)$ and $u\in A_{i+k_j,u}$ such that $su \in L(G_i) \parallel P_{k_j}( K_j)$. We will show that $su \in L(G_i) \parallel P_{k_j}( K_i)$ as well. First of all, $s\in P_{i+k_j}(K_i)$ implies that $P^{i+k_j}_{k_j}(s)\in  P^{i+k_j}_{k_j}P_{i+k_j}(K_i)=P_{k_j}( K_i)$. Note that $P^{i+k_j}_{k_j}(su)\in P^{i+k_j}_{k_j}(L(G_i) \parallel L(G_{k_j})) \subseteq  L(G_{k_j})$. Now we distinguish two cases (i) $u\in A_{k_j,u}$ and (ii) $u\in A_i\setminus A_{k_j,u}$.

    In the first case, we get $P^{i+k_j}_{k_j}(su)=P^{i+k_j}_{k_j}(s)u \in P_{k_j}( K_i)$, from controllability of $P_{k_j}( K_i)$ with respect to $L(G_{k_j})$ and $ A_{k_j,u}$ (item~\ref{cc1}). In the second case, we immediately get $P^{i+k_j}_{k_j}(su)=P^{i+k_j}_{k_j}(s)\in P_{k_j}( K_i)$. From $su \in L(G_i) \parallel P_{k_j}( K_j)$ we have that $su\in (P^{i+k_j}_{i})^{-1}(L(G_i))$. Thus, $su \in (P^{i+k_j}_{k_j})^{-1}(P_{k_j}( K_i))\cap  (P^{i+k_j}_{i})^{-1}(L(G_i))= L(G_i) \parallel P_{k_j}( K_i)$. Finally, by applying item~\ref{cc2} of two-level conditional controllability of $K_i$, we obtain that $su\in P_{k_j}(K_i)\subseteq P_{k_j}(\cup_{i\in I} K_i)$. Hence, item~\ref{cc2} of two-level conditional controllability holds for $\cup_{i\in I} K_i$.
  \end{proof}

  From Theorem~\ref{existence}, the supremal two-level conditional controllable sublanguage of a specification $K$ with respect to $A_{k_1},\dots A_{k_m}$, and $A_{u}$, denoted by $\suptwoCC(K, L, A_{i+k_j})$, always exists. Below we propose a procedure to compute the supremal two-level conditional controllable sublanguage $\suptwoCC(K, L, A_{i+k_j})$.

  Similarly as in the centralized coordination we introduce the following notation. For all $j=1,\dots,m$ and $i\in I_j$,
  \begin{equation}\label{sup2cc}
    \begin{aligned}
      \supC_{k_j}   & = \supC(P_{k_j}(K),L(G_{k_j}), A_{k_j,u})\\
      \supC_{i+k_j} & = \supC(P_{i+k_j}(K), L(G_i) \| \supC_{k_j}, A_{i+k_j,u})
    \end{aligned}
  \end{equation}
  where $\supC(K,L,A_u)$ denotes the supremal controllable sublanguage of $K$ with respect to $L$ and $A_u$, see~\cite{CL08}.

  Similarly as in the centralized coordination, the following inclusion always holds true.

  \begin{lem} \label{inclusion}
    For all $j=1,\dots ,m$ and for all $i\in I_j$, we have that 
    $P^{i+k_j}_{k_j}(\supC_{i+k_j})\subseteq \supC_{k_j}$. 
  \end{lem}
  \begin{proof}
    It follows from the definition of $\supC_{i+k_j}$ that $P^{i+k_j}_{k_j} (\supC_{i+k_j}) \subseteq \supC_{k_j}$, because $\supC_{k_j}$ is part of the plant language (over the alphabet $A_{k_j}$) of $ \supC_{i+k_j}$.
  \end{proof}
  
  Transitivity of controllability is needed later.
  \begin{lem}[\cite{KMvS14}]\label{lem_transC}
    Let $K\subseteq L \subseteq M$ be languages over $A$ such that $K$ is controllable with respect to $L$ and $A_u$, and $L$ is controllable with respect to $M$ and $A_u$. Then $K$ is controllable with respect to $M$ and $A_u$. 
    \hfill\QED
  \end{lem}

  \begin{theorem}\label{construction}
    Consider Problem~\ref{problem:m-controlsynthesis} and languages defined in~(\ref{sup2cc}). If $\cap_{i\in I_j} P^{i+k_j}_{k_j} (\supC_{i+k_j})$ is controllable with respect to $L(G_{k_j})$ and $A_{k_j,u}$, and if for all $j=1,\dots ,m$:
$P_k^{k_j}$ is an $L_{k_j}$-observer and OCC for $L_{k_j}$,  then $\suptwoCC(K, L, A_{i+k_j}) = \|_{j=1}^m \|_{i\in I_j} \supC_{i+k_j}$.
  \end{theorem}
  \begin{proof}
    For simplicity, we denote $\suptwoCC(K, L, A_{i+k_j})=\suptwoCC$, $M_j=\|_{i\in I_j}\supC_{i+k_j}$, and $M =\|_{j=1}^m M_j$. In order to show the inclusion $M \subseteq \suptwoCC$, it suffices to prove that for $j=1,\dots ,m$ and for $i\in I_j$, $P_{i+k_j}(\suptwoCC) \subseteq \supC_{i+k_j}$. Recall that $P_{i+k_j}(\suptwoCC)$ is controllable with respect to $L(G_i)\| P_{k_j}(\suptwoCC)$ and $A_{i+k_j,u}$, and $L(G_i)\| P_{k_j}(\suptwoCC)$ is controllable with respect to $L(G_i) \| \supC_{k_j}$ and $A_{i+k_j,u}$, because $P_{k_j}(\suptwoCC)$ being controllable with respect to $L(G_{k_j})$ is also controllable with respect to $\supC_{k_j}\subseteq L(G_{k_j})$. By the transitivity of controllability (Lemma~\ref{lem_transC}), $P_{i+k_j}(\suptwoCC)$ is controllable with respect to $L(G_i) \| \supC_{k_j}$ and $A_{i+k_j,u}$, which implies that $P_{i+k_j}(\suptwoCC) \subseteq \supC_{i+k_j}$.

    To prove the opposite inclusion, we show that $M$ is two-level conditionally controllable with respect to $G_i$, $i\in I_l$, and $G_{k_l}$, $l=1,\dots, m$. It follows from Lemma~\ref{cc2cc} that with our assumptions it is sufficient to show that for all $j=1,\dots ,m$, the languages $M_j$ are conditionally controllable with respect to $G_i$, $i\in I_j$, and $G_{k_j}$. For item 1, $P_{k_j}(M_j) = P_{k_j}(\|_{i\in I_j} \supC_{i+k_j})= \cap_{i\in I_j} P^{i+k_j}_{k_j} (\supC_{i+k_j})$ (by Lemma~\ref{lemma:Wonham} and the fact that $A_{k_j}$ contains all shared events of subsystems of the group $I_{j}$) is by assumption controllable with respect to $L(G_{k_j})$ and $A_{k_j,u}$.

    For item 2, we observe that for all $i\in I_j$, we get (the distributivity holds due to Lemma~\ref{lemma:Wonham}) $P_{i+k_j}(M_j) = P_{i+k_j}(\|_{i'\in I_j} \supC_{i'+k_j}) = \|_{i'\in I_j} P^{i+k_j}_{k_j} (\supC_{i'+k_j}) = \supC_{i+k_j} \parallel \|_{i'\in I_j}^{i\neq i'} P_{k_j} (\supC_{i'+k_j})$. Observe that for all $i\in I_j$, $P_{i+k_j}(M_j) = \supC_{i+k_j} \| P_{k_j}(M_j)$. This is because $\supC_{i+k_j} \| P_{k_j} (\|_{i'\in I_j} \supC_{i'+k_j}) = \supC_{i+k_j} \| \|_{i'\in I_j} P_{k_j} (\|_{i'\in I_j}\supC_{i'+k_j})= \supC_{i+k_j} \parallel \|_{i'\in I_j}^{i\neq i'} P_{k_j} (\supC_{i'+k_j})$. Therefore, $P_{i+k_j}(M_j) = \supC_{i+k_j} \| P_{k_j}(M_j) $ is controllable with respect to $[L(G_i) \| \supC_{k_j}] \parallel P_{k_j}(M) = L(G_i) \parallel P_{k_j}(M)$ using the fact that $P^{i+k_j}_{k_j} (\supC_{i+k_j}) \subseteq \supC_{k_j}$, for any $j=1,\dots ,m$ and $i\in I_j$, cf. Lemma~\ref{inclusion}. 
  \end{proof}

  Note that controllability of $\cap_{i\in I_j} P^{i+k_j}_{k_j} (\supC_{i+k_j})$ with respect to $L(G_{k_j})$ and $A_{k_j,u}$ for all $j=1,\dots ,m$, is not a suitable condition for verification. Clearly, for our prefix-closed languages, controllability of $P^{i+k_j}_{k_j} (\supC_{i+k_j})$ with respect to $L(G_{k_j})$ and $A_{k_j,u}$, $j=1,\dots ,m$ and $i\in I_j$, implies it. Moreover, two stronger checkable conditions are provided below. 
  
  It is easy to see that the equality in Lemma~\ref{inclusion} implies the sufficient condition of Theorem~\ref{construction}. Indeed, if for all $j=1,\dots ,m$ and for all $i\in I_j$, we have $P^{i+k_j}_{k_j} (\supC_{i+k_j}) \subseteq \supC_{k_j}$, then in particular $P^{i+k_j}_{k_j} (\supC_{i+k_j})$ is controllable with respect to $L(G_{k_j})$ and $A_{k_j,u}$. Hence, for all $j=1,\dots ,m$, $\cap_{i\in I_j} P^{i+k_j}_{k_j} (\supC_{i+k_j})$ is controllable with respect to $L(G_{k_j})$ and $A_{k_j,u}$, which proves the following result.
  
  \begin{corollary}\label{construction1}
    Consider the setting of Problem~\ref{problem:m-controlsynthesis} and the languages defined in~(\ref{sup2cc}). If for all $j=1,\dots ,m$:
$P_k^{k_j}$ is an $L_{k_j}$-observer and OCC for $L_{k_j}$, and for all $i\in I_j$, $P^{i+k_j}_{k_j}(\supC_{i+k_j}) = \supC_{k_j}$, then $\suptwoCC(K, L, A_{i+k_j}) = \|_{j=1}^m \parallel_{i\in I_j}\supC_{i+k_j}$. 
    \hfill\QED
  \end{corollary} 
  
  There is yet another sufficient condition
  that guarantees the controllability
  requirement in Theorem~\ref{construction}.
  Namely, local control consistency and observer properties
  (that are checkable by well-known methods).

  \begin{lem} \label{observerlcc}
    Let  for all $j=1, \dots ,m$ and for all $i\in I_j$, $P^{i+k_j}_{k_j}$ be an $(P^{i+k_j}_i)^{-1}L(G_i)$-observer and OCC for $(P^{i+k_j}_i)^{-1}L(G_i)$. Then for all $j=1, \dots ,m$, $\cap_{i\in I_j} P^{i+k_j}_{k_j} (\supC_{i+k_j})$ is controllable with respect to $L(G_{k_j})$ and $A_{k_j,u}$.
  \end{lem} 
  \begin{proof}
    The proof is the same as in~\cite{cdc2014}, the only difference is that here $A_{k_j}$, $j=1, \dots ,m$, replaces centralized $A_k$, and individual groups of subsystems on the low level replace all subsystems in the centralized coordination.
  %
  %
  \end{proof}

  We point out that even without the above conditions $\|_{j=1}^m \|_{i\in I_j}\supC_{i+k_j}$  is controllable with respect to $L$, but we cannot guarantee maximal permissiveness with respect to the two-level coordination control architecture.
%

\section{Decentralized supervisory control with communication}
\label{sec:dec} 
  In this section constructive results of the top-down coordination are applied to decentralized supervisory control with communicating supervisors. To avoid any confusion we use systematically $\Sigma$ to denote alphabets in decentralized control, while notation $A$ is reserved for alphabets in coordination control. Decentralized supervisory control is now briefly recalled.
 
\subsection{Decentralized Supervisory Control Problem}
\label{sec:bdscp} 
  Decentralized supervisory control differs from modular or coordination control,
  because global system is a large automaton without a product structure.
  In decentralized supervisory control 
  sensing and actuating capabilities are distributed among local supervisors $(S_i)_{i=1}^{n}$ 
  such that each $S_i$ observes a subset $\Sigma_{o,i}\subseteq \Sigma$
  and based on its observation it can disable its controllable  events $\Sigma_{c,i}$. Projections to locally observable events are denoted by 
  $P_i: \Sigma^* \to \Sigma_{o,i}^*$. 
  We use the notation $\Sigma_c=\cup_{i=1}^n \Sigma_{c,i}$, $\Sigma_o=\cup_{i=1}^n \Sigma_{o,i}$, $\Sigma_{u}=\Sigma\setminus \Sigma_{c}$, and $\Sigma_{uo}=\Sigma\setminus \Sigma_{o}$.

  Formally, a local supervisor $S_i$ for a generator $G$ is defined as a mapping $S_i : P_i(L(G))\to \Gamma$, where $\Gamma_i=\{\gamma\subseteq \Sigma \mid \gamma\supseteq (\Sigma\setminus \Sigma_{c,i})\}$ is the set of local control patterns, and $S_i(s)$ represents the set of locally enabled events when $S_i$ observes a string $s\in \Sigma_{o,i}^*$. Then the permissive local supervisor law is $S_i(s)=(\Sigma\setminus \Sigma_{c,i}) \cup \{ a\in \Sigma_{c,i}\mid \exists s'\in K \text{ with } P_i(s')=P_i(s) \text{ and } s'a\in K\}$. The global control law $S$ is given by conjunction of local ones: for $w\in \Sigma_{o,i}^*$, $S(w) =\cap_{i=1}^n S_i(P_i(w))$,
  This is why this control architecture is
  nowadays referred to as conjunctive and permissive.
  
  \begin{definition}\label{defcoop}
   $K\subseteq L$ is {\em $C\,\&\,P$ coobservable\/} with respect to $L=L(G)$ and $(\Sigma_{o,i})_{i=1}^{n}$ if for all $s\in K$, $a\in \Sigma_c$, and $sa \in L\setminus K$, there exists $i\in \{1,2,\ldots,n\}$ such that $a\in \Sigma_{c,i}$ and $(P_i^{-1}(P_i(s))\{a\}\cap K=\emptyset$.
    \hfill\QEDopen
  \end{definition}
  
  $C\,\&\,P$ coobservability can be interpreted in the following way: if we exit from the specification by  an event $a$, then there must exists at least one local supervisor that can control this event ($a\in \Sigma_{c,i}$) and can disable $a$
  unambiguously, i.e. all lookalike (for $S_i$) strings exit the specification as well.
  Recall that $C\,\&\,P$  coobservability can be decided in polynomial time~\cite{RW95}
  in the number of states of the  specification and system (but in
  exponential time in the number of local supervisors!). 
  
  Since the  counterpart of C\,\&\,P  coobservability, called D\,\&\,A  coobservability, is not studied in this paper, C\,\&\,P coobservability will be referred to as coobservability.

  It has been proved in~\cite{RW92} that controllability and coobservability  are the necessary and sufficient conditions to achieve a specification as the resulting closed-loop language. For languages that fail to satisfy these conditions it is important to compute a controllable and coobservable sublanguage. 
  
  We first recall that decomposability is strongly related to coobservability. 
  A language $K$ is {\em decomposable\/} with respect to alphabets $(\Sigma_{i})_{i=1}^{n}$ and $L$ if 
  $K = \|_{i=1}^{n} P_{i}(K) \cap L$.
  In the special case $L=\Sigma^*$ decomposability is called {\em separability\/}~\cite{GM04}, i.e.  $K$ is {\em separable\/} with respect to alphabets $(\Sigma_{i})_{i=1}^{n}$ if $K = \|_{i=1}^{n} P_{i}(K)$.
  Now we are ready to recall that (under the event set condition that will be shown
  non-restrictive in the following subsection) separability implies coobservability.
  \begin{theorem}[\cite{KM13}]\label{thm7}
    Assume that $\Sigma_{o,i}\cap \Sigma_c\subseteq \Sigma_{c,i}$, for $i=1,2,\ldots,n$. If $K$ is separable with respect to $(\Sigma_{o,i})_{i=1}^{n}$, then $K \cap L$ is coobservable with respect to $(\Sigma_{o,i})_{i=1}^{n}$ and $L$.
    \endproof
  \end{theorem}
    
%
%

\subsection{Construction of controllable and coobservable sublanguages based on two-level coordination control}
\label{sec:ccds}

  Now we show how the new constructive results of the top-down approach to multilevel coordination control from Section~\ref{sec:sup2cc} can be used to compute sublanguages that are by construction controllable and coobservable.

  The idea is that owing to our results of coordination control, the supremal conditionally controllable sublanguage (that is in particular controllable)  can be computed as a synchronous product of languages over alphabets enriched by communicated events from group coordinators. Hence the resulting language will be by construction not only controllable, but in view of Theorem~\ref{thm7} also coobservable. Moreover, due to the multilevel coordination, coordinator events will not be communicated among all subsystems, which was our first approach presented in~\cite{KM13} based on the centralized coordination. Roughly speaking, coordinator events from the centralized coordination are distributed to group  coordinator events, and these will only  be communicated among the subsystems belonging to the group. 
 
  Recall at this point that alphabets in decentralized control are denoted by $\Sigma$, while alphabets in coordination control are denoted by $A$.

  It follows from Theorem~\ref{construction} and its consequences that the supremal two-level conditionally-controllable sublanguage is decomposable with respect to alphabets $(A_{i+k_j}, \; j=1,\dots ,m, \; i\in I_j)$.
  
  Note that conditional decomposability with respect to alphabets $(\Sigma_{o,i})_{i=1}^{n}$ and $\Sigma_k$ such that $\cup_{i\neq j} (\Sigma_i\cap \Sigma_j)\subseteq \Sigma_k$, i.e.
  \[
      K = \|_{i=1}^{n} P_{i+k} (K),
  \]
  is nothing else but separability of $K$ with respect to
  $(\Sigma_{o,i}\cup \Sigma_k)_{i=1}^{n}$.
  An important feature of separability with respect to the alphabets
  of this form with intersection between all pairs of alphabets being equal to $\Sigma_k$
  is that it can be checked in polynomial time in the number of local agents~\cite{SCL12}.
 
  We recall that there always exists $\Sigma_k$ that makes language $K$ conditionally decomposable with respect to $(\Sigma_{i})_{i=1}^{n}$ and $\Sigma_k$~\cite{SCL12}. 
  Moreover, $K$ is conditionally decomposable if and only if there exist $M_i\subseteq \Sigma_{i+k}^*$ such that $K = \|_{i=1}^{n} M_i$.

  Consider now the setting of decentralized control with subsets of  events observable $(\Sigma_{o,i})_{i=1}^{n}$ and  controllable  $(\Sigma_{c,i})_{i=1}^{n}$ by agent (supervisor) $i$ and a specification $K\subseteq L=L(G)$. 
  
  We plunge the decentralized control problem into the coordination control problem by setting
  \[
    A_i = \Sigma_{o,i} \quad \text{ and }\quad A_{c,i} = \Sigma_{o,i}\cap \Sigma_{c,i}\,.
  \]
  For simplicity, the same notation for projection is kept: $P_i: A^* \to A_i^*$. The plant language $G$ will be over-approximated by a two-level modular plant $\|_{i=1}^{n}P_{i}(G)$, that is, by the parallel composition of projections to events observable by local control agents. We will organize the local supervisors into a two-level hierarchy. Similarly as in multilevel coordination control we will group the agents based on their interactions given by intersection of their alphabets, here observations $A_i = \Sigma_{o,i}$. The idea is to place agents with maximal shared observations to the same groups at the lowest level of the multilevel structure. Then in an ideal situation there will be no shared observations between different groups of agents, because all shared observations are realized within the low-level groups. This can be formalized by associating a square matrix with the number of shared events  observed by both and try to find after permutation  a block matrix structure such that the maximum of shared events is situated in the diagonal blocks, while off-diagonal blocks contain very small numbers (ideally zero matrices). Again, we denote $m$ low-level groups of agents by $I_j$, $j=1,\dots ,m\leq n$.

  According to the two-level top-down architecture we have to find an extension $A_k$ (a high-level coordinator alphabet) of these "high-level shared observations" $A_{sh}=\bigcup\nolimits_{k,\ell\in\{1, \dots ,m\}}^{k\not =l} (A_{I_k}\cap A_{I_\ell})$ such that $A_k\supseteq A_{sh}$ and $K = \|_{r=1}^m P_{I_r+k} (K)$. Similarly, we have to extend the shared low-level observations in groups, i.e. $A_{sh,r}=\bigcup\nolimits_{k,\ell\in I_r}^{k\neq \ell} (A_k\cap A_\ell)$, for $r=1,\dots,m$, to low-level coordinator alphabets  $A_{k,j}\supseteq A_{sh,r}$ such that $P_{I_r+k} (K) = \|_{j\in I_r} P_{j+k_r+k} (K)$, i.e. two-level conditional decomposability holds true.
 
  It is a common assumption in both modular and coordination supervisory control that shared events have the same controllability status in all components, i.e. $A_{i}\cap A_{c,j} \subseteq A_{c,i}$, for $i,j=1,2,\ldots,n$. Since  $A_i = \Sigma_{o,i}$, it is clear that this assumption is equivalent to $\Sigma_{o,i}\cap \Sigma_{c,j} \subseteq  \Sigma_{c,i}$, for $i,j=1,2,\ldots,n$. Hence, by Theorem~\ref{thm7}, separability implies coobservability and the condition stated therein is not restrictive.  More precisely we have:

  \begin{lem}\label{sharedevent_conditions}
    If $\Sigma_{o,i}\cap \Sigma_{c,j}\subseteq \Sigma_{c,i}$, for $i,j=1,2,\ldots,n$, then $(A_{i})_{i=1}^{n}$ and $(A_{c,i})_{i=1}^{n}$ defined above satisfy $A_{i}\cap A_{c,j}\subseteq A_{c,i}$, for all $i,j=1,2,\ldots,n$.
  \end{lem}
  \begin{proof}
    The assumption implies that for all $i,j=1,2,\ldots,n$, $A_{i}\cap A_{c,j} = \Sigma_{o,i}\cap  (\Sigma_{o,j}\cap \Sigma_{c,j})\subseteq \Sigma_{o,i}\cap \Sigma_{c,j} \subseteq \Sigma_{c,i}$ and, trivially, $A_{i}\cap A_{c,j} = \Sigma_{o,i}\cap  (\Sigma_{o,j}\cap \Sigma_{c,j}) \subseteq \Sigma_{o,i}$, hence $A_{i}\cap A_{c,j} \subseteq \Sigma_{c,i}\cap \Sigma_{o,i} = A_{c,i}$, which was to be shown.
  \end{proof}

  Once $A_k$ and $A_{k_j}$, $j=1,\dots ,m$, are found such that two-level conditional decomposability holds we compute languages $L_{k_j} = L(G_{k_j})$, $\supC_{k_j}$, and $\supC_{i+k_j}$ as in the previous section, cf. formula (\ref{sup2cc}), where $L(G_i)$ are replaced by $P_i(L)$. Namely, for all $j=1,\dots ,m$, $L_{k_j} = \|_{i=1}^n P_{k_j}(P_i(L))$, $\supC_{k_j}=\supC(P_{k_j}(K),L_{k_j}, A_{k_j,u})$, and for all $i\in I_j$, $\supC_{i+k_j}=\supC(P_{i+k_j}(K), P_i(L) \| \supC_{k_j}, A_{i+k_j,u})$.

  \begin{theorem} \label{main}
    Let $K\subseteq L$ be languages, and let $K$ be two-level conditionally decomposable with respect to $(A_{i})_{i=1}^{n}$, $A_{k,j}$, and $A_k$.  Then $\parallel_{j=1}^m \parallel_{i\in I_j}\supC_{i+k_j}$ is a sublanguage of $K$ controllable with respect to $L$ and $A_u$, and coobservable with respect to $L$ and $(A_{i+k_j})_{j=1,\dots,m; \; i\in I_j}$.  
    
    If $\|_{j=1}^m \cap_{i\in I_j} P^{i+k_j}_{k_j} (\supC_{i+k_j})$ is controllable with respect to $L(G_{k_j})$ and $A_{k_j,u}$, and for all $j=1,\dots ,m$:
$P_k^{k_j}$ is an $L_{k_j}$-observer and OCC for $L_{k_j}$, then $\|_{j=1}^m \|_{i\in I_j}\supC_{i+k_j}=\suptwoCC(K, L, A_{i+k_j})$ is the largest controllable and coobservable language we can obtain using the two-level coordination.
  \end{theorem}
  \begin{proof}
    First of all, $M=\|_{j=1}^m \|_{i\in I_j} \supC(P_{i+k_j}(K)$, $P_i(L) \| \supC_{k_j}, A_{i+k_j,u})$ is separable with respect to alphabets $(A_i\cup A_{k_j})_{j=1,\dots,m; \; i\in I_j}$ as a composition of languages over these alphabets. Hence, by Theorem \ref{thm7}, $M \cap L$ is coobservable with respect to the same alphabets and $L$.  Note that $M \cap L=M$, because $M\subseteq K$ due to conditional decomposability of $K$ and $K \subseteq L$, whence $M\subseteq L$.

    $M$ is also controllable with respect to $L$. This is because, by Lemma~\ref{feng}, $M=\|_{j=1}^m \|_{i\in I_j}\supC_{i+k_j}$ is always controllable with respect to $\|_{j=1}^m \|_{i\in I_j} P_i(L) \| \supC_{k_j}$, and since $\supC_{k_j}$ is controllable with respect to $L(G_{k_j})$, we get by transitivity of controllability that $M$ is controllable with respect to $\|_{j=1}^m \|_{i\in I_j} P_i(L) \| L_{k_j} = \|_{i=1}^n P_{i}(L) \|_{j=1}^m \|_{i\in I_j} P_i(L) = \|_{i=1}^n P_{i}(L)$, because from Lemma~\ref{obvious} we have $P_{i}(L) \parallel P_{k_j}(P_i(L)) = P_{i}(L)$. Since $L\subseteq \|_{i=1}^n P_{i}(L)$, $M$ is also controllable with respect to $L$. The rest of the claim is a consequence of Theorem \ref{construction} and the above described translation of the decentralized control problem into a coordination control problem. 
  \end{proof}
  
  Note that according to the above results there are stronger sufficient conditions that are more suitable for verification, namely that for $j=1, \dots ,m$ and $i\in I_j$, either $P^{i+k_j}_{k_j}(\supC_{i+k_j})=\supC_{k_j}$, or $P^{i+k_j}_{k_j}$ are $(P^{i+k_j}_i)^{-1}P_i(L)$-observers and OCC for $(P^{i+k_j}_i)^{-1}P_i(L)$. The first condition has the advantage that if it holds, our results can be extended to the non-prefix-closed case. The second condition has the advantage that there are known algorithms to extend the local alphabets such that corresponding projections become OCC and satisfy the observer property. 
  Finally, let us point out that without any assumption (except two-level conditional decomposability) $M$ might be a too small language, because without the above additional conditions we cannot guarantee the optimality with respect to the coordination control ($\suptwoCC(K, L, A_{i+k_j})$) and the optimality is lost twice (it is potentially lost due to over-approximation of $L$ by $\|_{i=1}^n P_{i}(L)$ unless $L$ is decomposable).

\subsection{Example}
  Let the languages $K$ and $L$ be given by generators on Figs.~\ref{figK} and~\ref{figL}, respectively. The alphabets of local agents are $\Sigma_{o,1}=\{a,b,u_1,u\}$, $\Sigma_{o,2}=\{a,b,u_2,u\}$, $\Sigma_{o,3}=\{v,b,v_1,b_1\}$, $\Sigma_{o,4}=\{v,b,v_2,b_2\}$, $\Sigma_{c,1}=\Sigma_{c,2}=\{a,b\}$, $\Sigma_{c,3}=\{v,v_1,b_1\}$, $\Sigma_{c,4}=\{v,v_2,b_2\}$, $\Sigma_u=\{b,u,u_1,u_2\}$, and $\Sigma_c=\Sigma\setminus \Sigma_u$. 

  Note that $K$ is not controllable with respect to $L$, because e.g. $v_2v_1b\in K\Sigma_u \cap L$, but $v_2v_1b\not\in K$. On the other hand, $K$ is coobservable with respect to $L$ and $\Sigma_{o,i}$, $i=1,2,3,4$. However, $N=\supC(K,L,\Sigma_u)$ is not coobservable with respect to $L$ and $\Sigma_{o,i}$, $i=1,2,3,4$, anymore. This is because both $v_1v_2\in L$ and $v_2v_1\in L$, $v_1v_2\in N$, $v_2\in N$, while $v_2v_1\not\in N$. 
  
  It is clear that at least one of agents 3 and 4 has to observe both $v_1$ and $v_2$ in order to issue a correct control decision. This means that it is not clear how to compute a sublanguage that is at the same time controllable and coobservable. Also notice that if $b$ were controllable, $K$ would be controllable, but still not coobservable. We now use our two-level coordination control approach.
  \begin{figure}
    \centering
    \includegraphics[scale=.22]{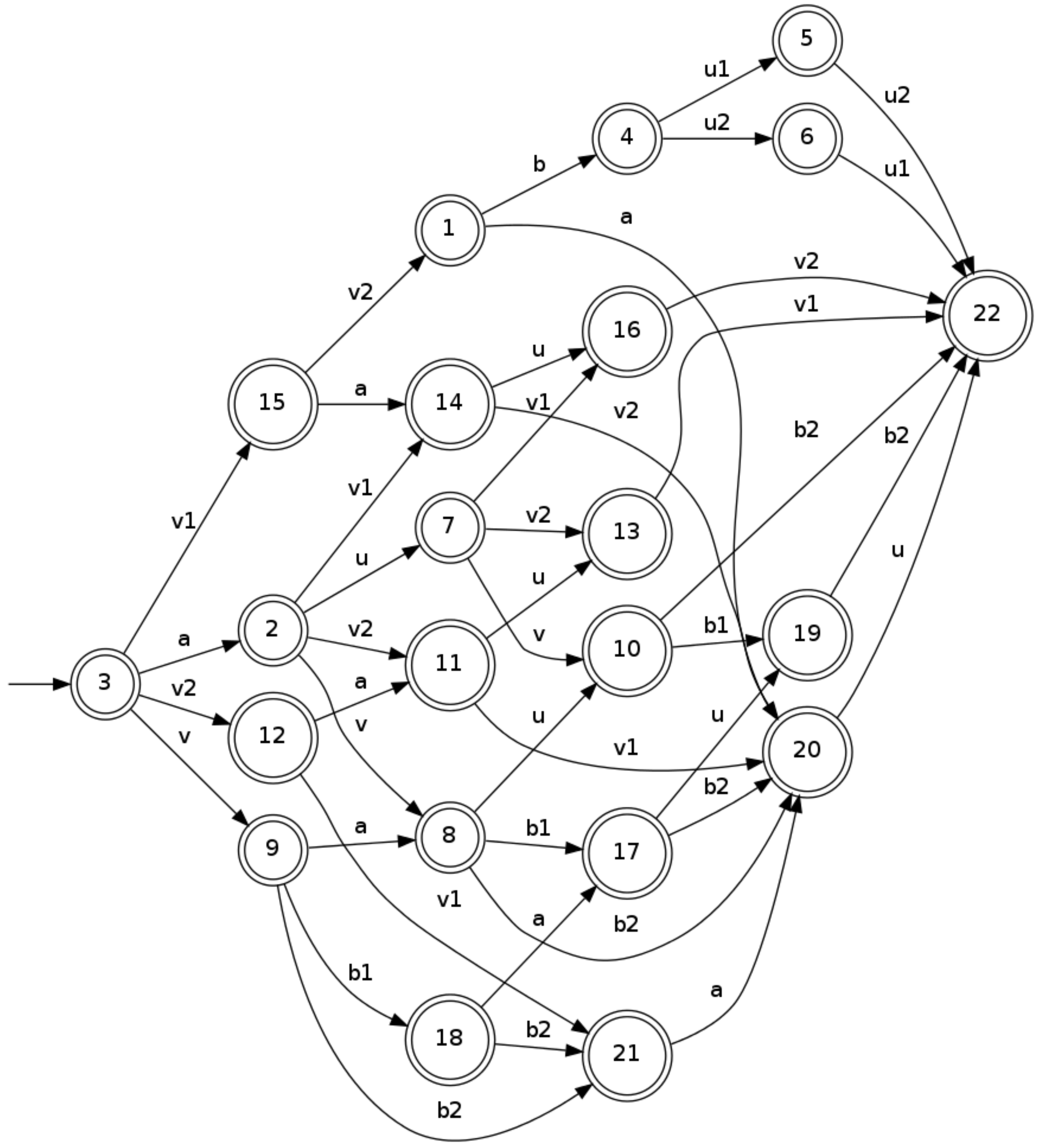}
    \caption{Specification $K$}
    \label{figK}
  \end{figure}
  \begin{figure}
    \centering
    \includegraphics[scale=.22]{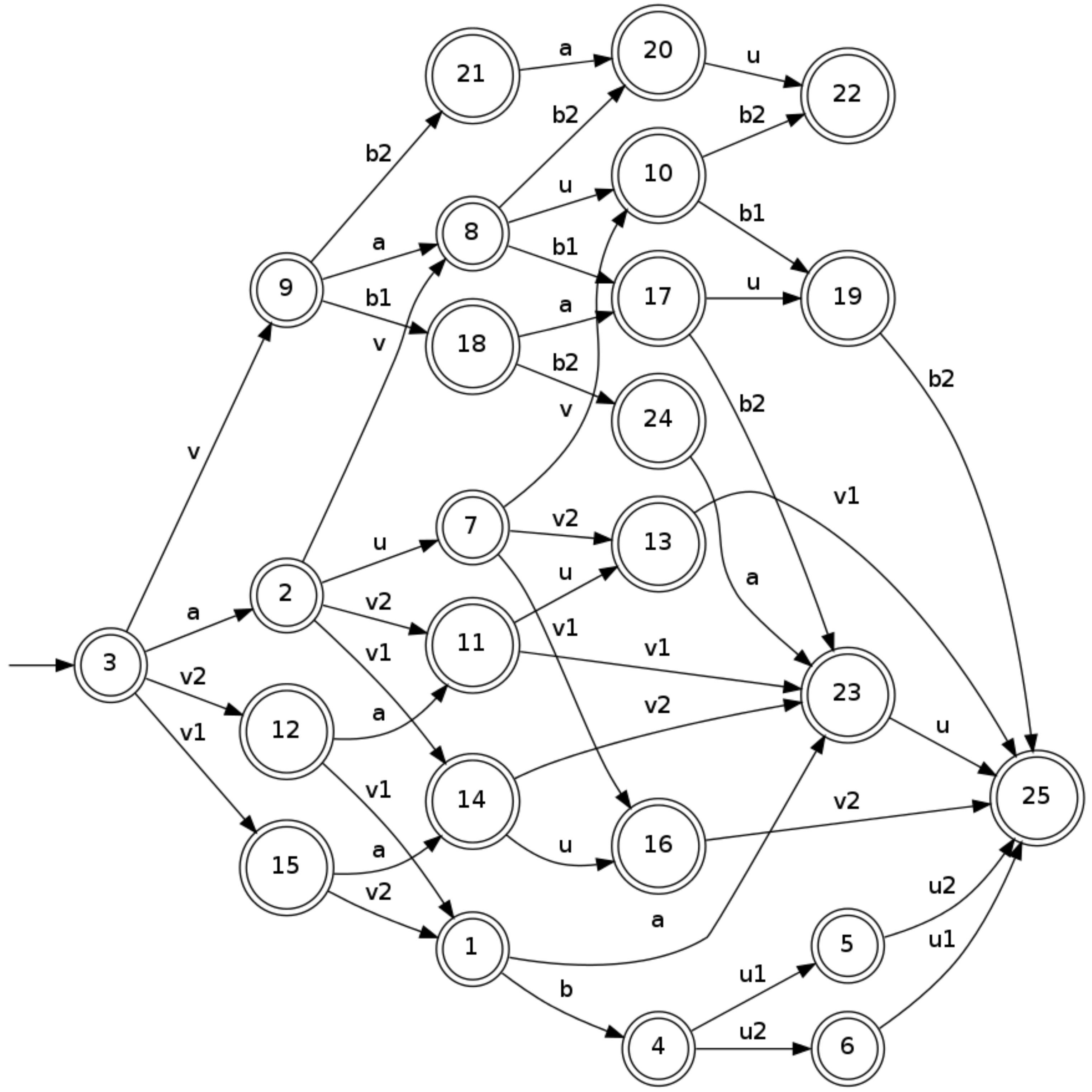}
    \caption{Plant $L$}
    \label{figL}
  \end{figure}
  
  First, we project the plant to the alphabets of observable events, cf. Fig.~\ref{figA}. The inclusion $L\subseteq \|_{i=1}^n P_{i}(L)$ is strict. We set $A_i=\Sigma_{o,i}$ for $i=1,2,3,4$. Now we have to extend $A_{sh}= (A_{I_1}\cap A_{I_2})=(A_1\cup A_2)\cap (A_3\cup A_4)=\{b\}$ such that for $A_k\supseteq A_{sh}$, $K=P_{1+2+k}(K)\| P_{3+4+k}(K)$. It turns out that no extension of $A_{sh}$ is needed, i.e. $A_k=\{b\}$.
  
  We now need to find low-level coordination alphabets $A_{k_1}$ and $A_{k_2}$ such that  conditions on the low-level for two-level conditionally decomposability hold true. Since $P_{1+2+k}(K)=P_1(K)\|P_2(K)$, there is no need to extend $A_{sh,1}=A_1 \cap A_2=\{a,u,b\}$ and take $A_{k_1}=A_{sh,1}$. Therefore, $P_{1+2+k}(K)=P_{1+k_1}(K)\|P_{2+k_1}(K)$,  with $A_{k_1}=\{a,u,b\}$. In particular, coordinator $L_{k_1}$ is not needed. Secondly, $P_{3+4+k}(K)$ is not decomposable with respect to  $A_{3+k}$ and $A_{4+k}$. Hence, we have to extend $A_k$ by $A_{k_2}$ to make the equation $P_{3+4+k}(K)=P_{3+k_2}(K)\|P_{4+k_2}(K)$ hold true. We choose $A_{k_2}=\{v_1, b_2, v, b\}$ and it can be checked that $P_{3+4+k}(K)$ is conditionally decomposable with respect to  alphabets $A_{3}$, $A_4$, and $A_{k_2}$. The corresponding coordinator for the second group is $L_{k_2}=\|_{i=1}^n P_i(L)=\{\eps,v, vb_2,v_1, v_1b\}$.
 
  Now starts the actual computation of a sublanguage of $K$ that is both controllable with respect to $L$ and $A_u$ and coobservable with respect to $L$ and observable alphabets $A_{1+k_1,o}$, $A_{2+k_1,o}$, $A_{3+k_2,o}$, $A_{4+k_2,o}$. Note that $K$ is not two-level conditionally controllable, hence the constructive procedure is needed. In fact, $P_{4+k_2}(K)$ given in Fig.~\ref{P3+k_2K} is not controllable with respect to $P_4(L) \| P_{k_2}(K)$ given in Fig.~\ref{L4Pk_2K}. Therefore, corresponding supervisor $S_4$ for $P_4(L)$ is not simply given by $P_{4+k_2}(K)$, but has to be computed. Its language is given by $\supC_{4+k_2}=\supC(P_{4+k_2}(K),P_4(L)\| \supC_{k_2}, A_{u,4+k_2})$. First we compute the supervisor $\supC_{k_2}$ for coordinator $L_{k_2}$. We have $\supC_{k_2}=\supC(P_{k_2}(K),L_{k_2}, A_{u,k_2})=P_{k_2}(K)=L_{k_2}$ computed above. Thus, $\supC_{4+k_2}=\{\overline{v_1,v_2,b,v_2,vb_2}\}$ is given by Fig.~\ref{L4Pk_2K}. This supervisor simply disables $v_1$ after $v_2$ has occurred. Similarly, we have to derive the supervisor $\supC_{3+k_2}$ for $P_3(L)$. We note that  $P_{3+k_2}(K)$, given in Fig.~\ref{P3+k_2K}, is controllable with respect to $P_3(L) \| P_{k_2}(K)$. Thus, $\supC_{3+k_2}=P_{3+k_2}(K)$. Concerning the first group of agents, no control action is required for supervisors  $\supC_{1+k_1}$ and  $\supC_{2+k_1}$, because $P_{1+2}(K)=P_1(L)\|P_2(L)$. Moreover, $\supC_{1+k_1}=P_{1+k_1}(K)=P_1(K)=P_1(L)$ and $\supC_{2+k_1}=P_{2+k_1}(K)=P_2(K)=P_2(L)$.

  \begin{figure}
    \centering
    \includegraphics[scale=.75]{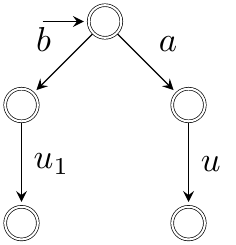}
    \includegraphics[scale=.75]{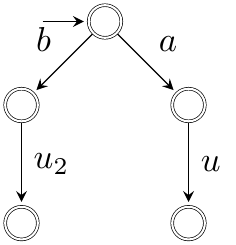}
    \includegraphics[scale=.75]{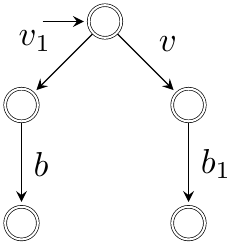}
    \includegraphics[scale=.75]{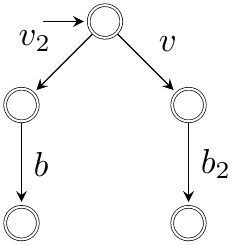}
    \caption{Projections $P_1(L),\dots,P_4(L)$}
    \label{figA}
  \end{figure}

Since all conditions in Theorem~\ref{construction} are satisfied, 
$\supC_{1+k_1} \| \supC_{2+k_1} \| \supC_{3+k_2} \| \supC_{4+k_2} = \suptwoCC(K, L, A_{i+k_j})$. In particular, it is controllable with respect to $L$ and $A_u$ and coobservable with respect to $L$ and $A_{1+k_1,o}$, $A_{2+k_1,o}$, $A_{3+k_2,o}$, $A_{4+k_2,o}$.
  \hfill$\triangleleft$
  
  \begin{figure}
    \centering
    \includegraphics[scale=.8]{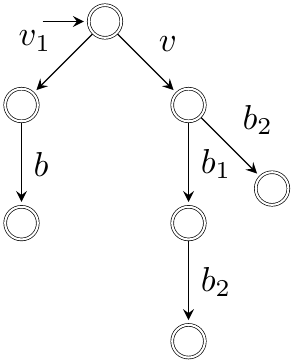} 
    \includegraphics[scale=.8]{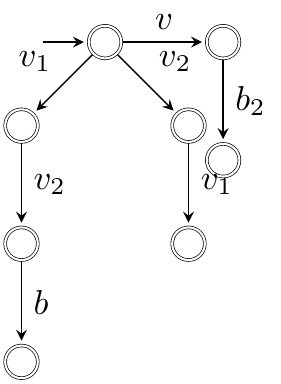} 
    \caption{$P_{3+k_2}(K)$ and $P_{4+k+k_2}(K)$}
    \label{P3+k_2K}
  \end{figure} 
  
  \begin{figure}
    \centering
    \includegraphics[scale=.8]{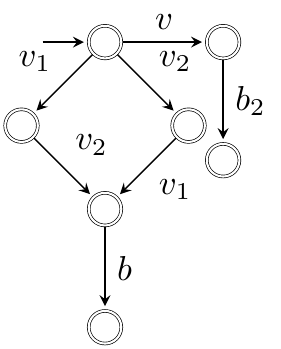} 
    \includegraphics[scale=.8]{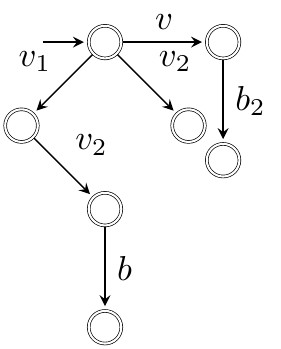} 
    \caption{ $L_4 \| P_{k_2}(K)$ and $\supC_{4+k_2}$}
    \label{L4Pk_2K}
  \end{figure}

\section{Concluding remarks}\label{conclusion}
  We have extended existential results of multilevel coordination control to constructive results with a procedure to compute the supremal sublanguage achievable in the two-level architecture. The constructive results have been applied to decentralized supervisory control with communicating supervisors, where local supervisors communicate with each other via a coordinator of the group.
 

\section{Acknowledgments}
  The authors would like to thank St\'ephane Lafortune and Feng Lin for a very fruitful discussion during their visit at the University of Michigan and Wayne State University. 
  
  The research was supported by the M\v{S}MT grant LH13012 (MUSIC), and by RVO: 67985840. 
  
\bibliographystyle{IEEEtranS}
\bibliography{multimost}

\end{document}